\documentclass[letterpaper,10pt,twocolumn,twoside,journal]{IEEEtran} 

\usepackage{times}
\usepackage{cite}
\usepackage{amssymb}
\usepackage{mathtools}
\usepackage{graphicx}
\usepackage{graphicx,xcolor}
\usepackage{algorithm}
\usepackage{hyperref}
\usepackage[noend]{algpseudocode}
\usepackage{subfigure}

\usepackage[font =small]{caption}

\usepackage{amsthm}

\usepackage{tikz,calc,bbding}
\usetikzlibrary{shapes, arrows, decorations, decorations.pathreplacing, decorations.pathmorphing, decorations.markings, fit, matrix}
\usepackage{arydshln}

\newtheorem{theorem}{Theorem}[section]
\newtheorem{proposition}[theorem]{Proposition}
\newtheorem{lemma}[theorem]{Lemma}

\newtheorem{remark}[theorem]{Remark}

\newtheorem{assumption}[theorem]{Assumption}

\newcommand{\longthmtitle}[1]{\mbox{}{\textit{(#1):}}}

\newcommand{\real}{\ensuremath{\mathbb{R}}}

\newcommand{\integers}{{\mathbb{Z}}}
\newcommand{\intpos}{{\mathbb{N}}}

\newcommand{\intnonneg}{{\mathbb{Z}}_{\ge 0}}

\newcommand{\setdef}[2]{\{#1 \; | \; #2\}}

\newcommand{\Ac}{\mathcal{A}}

\newcommand{\Ec}{\mathcal{E}}

\newcommand{\Gc}{\mathcal{G}}
\newcommand{\Hc}{\mathcal{H}}

\newcommand{\Kc}{\mathcal{K}}
\newcommand{\Lc}{\mathcal{L}}

\newcommand{\Nc}{\mathcal{N}}

\newcommand{\Qc}{\mathcal{Q}}

\newcommand{\Sc}{\mathcal{S}}

\newcommand{\Vc}{\mathcal{V}}

\newcommand\col{\textnormal{col}}

\newcommand{\diag}{{\rm diag}}

\newcommand{\trans}{\mathsf{T}}

\newcommand\blkdiag{\text{blkdiag}}
\newcommand{\norm}[1]{\left\lVert#1\right\rVert}

\usepackage{enumitem}

\newcommand{\rowrank}{\textnormal{rowrank }}

\newcommand{\oprocendsymbol}{\hbox{$\square$}}
\newcommand{\oprocend}{\relax\ifmmode\else\unskip\hfill\fi\oprocendsymbol}

\allowdisplaybreaks

\renewcommand{\algorithmicrequire}{\textbf{Input:}}

\title{\LARGE \bf Data-Based Receding Horizon Control of
  Linear Network Systems\thanks{This work was supported by
    ARO-W911NF-18-1-0213.}}  \author{Ahmed Allibhoy \quad Jorge
  Cort\'es \thanks{A.  Allibhoy and J. Cort\'es are with the
    Department of Mechanical and Aerospace Engineering, UC San Diego,
    \{aallibho,cortes\}@ucsd.edu}}

    \def\BibTeX{{\rm B\kern-.05em{\sc i\kern-.025em b}\kern-.08em
    T\kern-.1667em\lower.7ex\hbox{E}\kern-.125emX}}
\begin{document}

\maketitle
\thispagestyle{empty}

\begin{abstract}%
  We propose a distributed data-based predictive control scheme to
  stabilize a network system described by linear dynamics.  Agents
  cooperate to predict the future system evolution without knowledge
  of the dynamics, relying instead on learning a data-based
  representation from a single sample trajectory. We employ this
  representation to reformulate the finite-horizon Linear Quadratic
  Regulator problem as a network optimization with separable objective
  functions and locally expressible constraints. We show that the
  controller resulting from approximately solving this problem using a
  distributed optimization algorithm in a receding horizon manner is
  stabilizing. We validate our results through numerical simulations.
\end{abstract}


\section{Introduction}\label{sec:intro}

\IEEEPARstart{W}{ith} the growing complexity of engineering systems,
data-based methods in control theory are becoming increasingly
popular, particularly for systems where it is too difficult to develop
models from first principles and parameter identification is
impractical or too costly. An important class of such systems are
network systems, which arise in many applications such as
neuroscience, power systems, traffic management, and robotics.
Without a system model, agents must use sampled data to characterize
the network behavior. However, the decentralized nature of the system
means that agents only have access to information that can be measured
locally, and must coordinate with one another to predict the network
response and decide their control actions. These observations motivate
the focus here on distributed data-based control of network systems
with linear dynamics.

\paragraph*{Literature Review}
Distributed control of network systems is a burgeoning area of
research, see e.g.,~\cite{FB-JC-SM:08cor, MM-ME:10, RRN-JMM:14} and
references therein.  In general, designing optimal controllers for
network systems is an NP-hard problem, but under certain conditions
optimal distributed controllers for linear systems can be obtained as
the solution to a convex program~\cite{MR-SL:06}.  When these
conditions do not hold, suboptimal controllers can be obtained by
convex relaxations \cite{FL-MF-MRJ:13, GF-RM-AK-JL:17} or convex
restrictions \cite{LF-YZ-AP-MK:19} of the original problem. Although
these methods produce distributed controllers, the computation of the
controller itself is typically done offline, in a centralized manner,
and requires knowledge of the underlying system model.  Reinforcement
learning (RL) is an increasingly popular approach for controlling
robots \cite{JK-AJB-JP:13} and multi-agent systems
\cite{LB-RB-BDS:08}. However, RL approaches typically require a very
large number of samples to perform effectively \cite{BR:19} and their
complexity makes it difficult to get stability, safety, and robustness
guarantees as is standard with other control approaches.  For
applications where safety assurances are required, model predictive
control (MPC) is widely used since performance and safety constraints
can be directly incorporated into an optimization problem that is
solved online. Several distributed MPC formulations are available for
multi-agent systems where the dynamics of the agents are coupled, such
as \cite{DJ-BK:02, WBD:07} where each agent implements a control
policy minimizing its own objective while accounting for network
interactions locally, or \cite{BTS-ANV-JBR-SJ-GP:10} where agents
cooperate to minimize a system-wide objective using a network
optimization algorithm.  Data-based approaches to predictive control
have also been proposed.  System identification~\cite{LL:99} is often
leveraged to learn a parameterized model which can then be used with
any of the MPC formulations previously mentioned. Methods for
implementing a controller directly from sampled data without any
intermediate identification also exist. The fundamental lemma from
behavioral systems theory \cite{JCW-PR-IM-BLMDM:05}, which
characterizes system trajectories from a single sample trajectory, has
recently gained attention in the area of data-based control
\cite{USP-MI:09, TMM-PR:17, CDP-PT:19}, and has been used for
predictive control in the recently developed DeePC
framework~\cite{JC-JL-FD:19, JB-JK-MAM-FA:19}.  Our work here extends
the DeePC framework to network systems where each node only has
partial access to the~data.

\paragraph*{Statement of Contributions}
We develop distributed data-based feedback controllers for network
systems\footnote{Throughout the paper, we make use of the following
  notation. Given integers, $a, b \in \integers$ with $a < b$, let
  $[a, b] = \{ a, a + 1, \dots, b \}$. Let $\Gc = (\Vc, \Ec)$ be an
  undirected graph with $N$ nodes, where $\Vc = [1, N]$ and $\Ec
  \subset \Vc \times \Vc$. The neighbors of $i \in \Vc$ are $\Nc_i =
  \{ j : (i, j) \in \Ec\}$. Given $S = \{ s_1, s_2, \dots, s_M \}
  \subseteq [1, N]$ and a vector $x = [x_1^\trans, x_2^\trans, \dots,
  x_N^\trans ]$, we denote $x_S = \begin{bmatrix} x_{s_1}^\trans &
    x_{s_2}^\trans & \cdots & x_{s_M}^\trans
\end{bmatrix}$. For $x_i \in \real^{d_i}$ with $i \in [1,K]$, we let
$\col(x_1, x_2, \dots, x_K) = [x_1^\trans, x_2^\trans, \dots,
x_K^\trans]^\trans$.  For positive semidefinite $Q \in \real^{n \times
n}$, we denote $\norm{x}_Q = \sqrt{x^\trans Q x}$.  For $M \in
\real^{n \times m}$, we denote by $M^\dagger$ its Moore-Penrose
pseudoinverse.  The Hankel matrix of a signal $w:[0, T] \to \real^k$
with $t \le T$ block rows is the~$kt \times (T - t + 1)$ matrix
\[  
\Hc_t(w) = \begin{bmatrix}
w(0)     & w(1)   & \cdots & w(T - t) \\
w(1)     & w(2)   & \cdots & w(T - t + 1) \\
\vdots   & \vdots & \ddots & \vdots \\
w(t - 1) & w(t) & \cdots & w(T - 1)
\end{bmatrix} .
\]
Given two signals $v_1:[0, T - 1] \to \real^{k_1}$ and $v_2:[0, T - 1]
\to \real^{k_2}$, let $v = \col(v_1, v_2)$ be the signal where $v(t) =
v_1(t)$ for $0 \leq t < T$, and $v(t) = v_2(t - T)$ for $T \leq t <
2T$.}.  A group of agents whose state evolves according to unknown
coupled linear dynamics each have access to their own state and those
of their neighbors in some sample trajectory. Their collective
objective is to drive the network state to the origin while minimizing
a quadratic objective function without knowledge of the system
dynamics. The approach we use computes the control policy online and
in a distributed manner by extending the DeePC formalism to the
network case.  Building upon the fundamental lemma, we introduce a new
distributed, data-based representation of possible network
trajectories. We use this representation to pose the control synthesis
as a network optimization problem, without state or input constraints,
in terms of the data available to each agent.  We show that this
optimization problem is equivalent to the standard finite-horizon
Linear Quadratic Regulation (LQR) problem and introduce a primal-dual
method along with a suboptimality certificate to allow agents to
cooperatively find an approximate solution.  Finally, we show that the
controller that results from implementing the distributed solver in a
receding horizon manner is stabilizing.

\section{Preliminaries}\label{sec:preliminaries}

We briefy recall here basic concepts on the identifiability of Linear
Time-Invariant (LTI) systems from data. Given $t, T_d \in \intnonneg$ with
$t < T_d$, a signal $w:[0, T_d - 1] \to \real^k$ is \textit{persistently
  exciting of order} $t$ if $\rowrank \Hc_t(w) = kt$.  Informally,
this means that any arbitrary signal $v:[0, t - 1] \to \real^k$ can be
described as a linear combination of windows of width $t$ in the
signal $w$. A necessary condition for persistence of excitation is $T_d
\geq (k + 1)t - 1$.

\begin{lemma}\longthmtitle{Fundamental
    Lemma~\cite{JCW-PR-IM-BLMDM:05}}\label{lem:fundamental}
  Consider the LTI system $x(t + 1) = Ax(t) + Bu(t)$, with $(A, B)$
  controllable. Let $u^d:[0, T_d - 1] \to \real^m$, $x^d:[0, T_d - 1] \to \real^n$
  be sequences such that $w^d = \col(u^d, x^d)$ is a trajectory of the
  system and $u^d$ is persistently exciting of order $n + \tau$. Then
  for any pair $u:[0, \tau - 1] \to \real^m$, $x:[0, \tau - 1] \to \real^n$,
  $w = \col(u, x)$ is a trajectory of the system if and only if there
  exists $g \in \real^{T_d -\tau - 1}$ such that $\Hc_\tau(w^d)g =w$.
\end{lemma}

Lemma~\ref{lem:fundamental} is stated here in state-space form, even
though the result was originally presented in the language of
behavioral systems theory.  The result states that all trajectories of
a controllable LTI system can be characterized by a single trajectory
if the corresponding input is persistently exciting of sufficiently
high order, obviating the need for a model or parameter estimation
when designing a controller.

\section{Problem Formulation}\label{sec:problem}
Consider a network system described by an undirected graph $\Gc =
(\Vc, \Ec)$ with $N$ nodes. Each node corresponds to an agent with
sensing, communication, and computation capabilities. Each edge
corresponds to both a physical coupling and a communication link
between the corresponding agents.  A subset of the nodes $S \subset
\Vc$, with $|S| = M$, also have actuation capabilities via inputs
$u_{i} \in \real^{m_{i}}$. The system dynamics are~then
\begin{equation}\label{eq:dynamics}
  \small
  x_i(t + 1) = \left\{
    \begin{aligned}
    &A_{ii}x_i(t) + \sum_{j \in \Nc_i} A_{ij}x_j(t) + B_{i}u_{i} &
    i \in S ,\cr
    &A_{ii}x_i(t) + \sum_{j \in \Nc_i} A_{ij}x_j(t) & i \notin S ,
    \end{aligned}\right.
\end{equation}
where $x_i \in \real^{n_i}$, $A_{ij} \in \real^{n_i \times n_j}$ and
$B_{i} \in \real^{n_i \times m_{i}}$. Let $n = \sum_{i=1}^{N}n_i$
and $m = \sum_{i=1}^{M}m_i$ and define $x = \col(x_1, x_2, \cdots,
x_N) \in \real^n$ and $u = \col(u_1, u_2, \cdots, u_M) \in
\real^m$. Let $A \in \real^{n \times n}$ and $B \in \real^{n \times
  m}$ be matrices so that~\eqref{eq:dynamics} takes the compact form
$x(t + 1) = Ax(t) + Bu(t)$.

To each node $i \in \Vc$, we associate an objective of the form
$J_i(x_i, u_{i}) = \sum_{t=0}^{T - 1} \norm{x_i(t)}_{Q_i} +
\norm{u_{i}(t)}_{R_{i}}$ when $i \in S$ and $J_i(x_i) = \sum_{t=0}^{T
  - 1} \norm{x_i(t)}_{Q_i}$ otherwise.  Here, each $Q_i \in \real^{n_i
  \times n_i}$ is positive semidefinite, each $R_{i} \in \real^{i
  \times i}$ is positive definite, $\col(u, x)$ is a system
trajectory, and $T$ is the time horizon of trajectories being
considered.

Each node wants to drive its state $x_i$ to the origin while
minimizing~$J_i$ and satisfying the constraints.  The resulting
network objective function is the sum of the objective functions
across the nodes. Letting $Q = \blkdiag(Q_1, Q_2, \dots, Q_N)\in
\real^{n \times n}$ and $R = \blkdiag(R_{s_1}, R_{s_2}, \dots,
R_{s_M}) \in \real^{m \times m}$, this objective can be written as
\begin{equation*}%\label{eq:objective}
  \begin{aligned}
    J(x, u) &= \sum_{i \in S}
    J_i(x_i, u) + \sum_{i \in \Vc \setminus S} J_i(x_i)
  \end{aligned}
\end{equation*}
so that $J(x, u) = \sum_{t=0}^{T - 1}\norm{x(t)}_{Q} +
\norm{u(t)}_R$. If the system starts from $x(0) = x^0 \in \real^n$,
the agents' goal can be formulated as the network optimization
problem:
\begin{align}\label{eq:lqr}
  &\underset{u, x}{\text{minimize}} & & \sum_{t=0}^{T -
    1}\norm{x(t)}_{Q} + \norm{u(t)}_R
  \\
  \notag & \text{subject to} & & x(t + 1) = Ax(t) + Bu(t), &\text{for
    $t \in [0,T_{\text{lqr}}]$}
  \\
  \notag & & & x(0) = x^{0}, \; x(T) = 0.
\end{align}
Note that the agents' decisions on their control inputs are coupled
through the constraints. Since $R \succ 0$, if~\eqref{eq:lqr} is
feasible, its optimal state and input trajectories are unique.

A key aspect of this paper is that we consider scenarios where the
system matrices $A$ and $B$ are unknown to the network. Instead, we
assume that, for a set of given input sequences $\{u^{d}_{i}:[0, T_d
-1] \to \real^{m_{i}}\}_{i \in S}$, the corresponding state
trajectories $\{x^{d}_{i}:[0, T_d-1] \to \real^{n_i}\}_{i \in \Vc}$
are available, and each node $i \in \Vc$ has access to its own state
trajectory as well as those of its neighbors. Actuated nodes $i\in S$
also have access to their own input $u^d_{i}$, but this is unknown to
its neighbors~$ \Nc_i$. Our aim is to synthesize a control policy that
can be implemented by each node in a distributed way with data
available to it. The resulting controller should stabilize the system
to the origin while minimizing~$J(x, u)$.

\section{Data-Based Representation for
  Optimization}\label{sec:results}

Here, we introduce a data-based representation of system
trajectories that is employed to pose a network optimization problem
equivalent to~\eqref{eq:lqr}.  Throughout this section, we let
$x^d:[0, T_d - 1] \to \real^n$, $u^d:[0, T_d - 1] \to \real^m$ be
sequences such that $w^d(t) = \col(u^d(t), x^d(t))$ is a trajectory of
\eqref{eq:dynamics}.  Let
\[ 
w^d_i(t) =
\begin{cases}
  \col(u^d_{i}(t), x^d_{\Nc_i}(t), x^d_{i}(t)) &\text{if $i \in S$}
  \\
  \col( x^d_{\Nc_i}(t), x^d_{i}(t)) & \text{if $i \notin S$}
\end{cases},
\]
for each $i \in \Vc$ and $0 \leq t < T_d - 1$. Then $w_i^d$ is the
data available to each node. Let $u:[0, \tau] \to \real^m$, $x:[0,
\tau] \to \real^n$ be arbitrary sequences where $T_d \geq (n + m)\tau
- 1$. Define $w(t) = \col(x(t), u(t))$ and
\[
w_i(t) = \begin{cases}
  \col(u_{i}(t), x_{\Nc_i}(t), x_{i}(t)) &\text{if $i \in S$}
  \\
  \col( x_{\Nc_i}(t), x_{i}(t)) & \text{if $i \notin S$}
\end{cases}. 
\]
Let $k_i = n_i \tau + \sum_{j \in \Nc_i}n_j\tau + m_i\tau$ for $i \in
S$, and $k_i = n_i\tau + \sum_{j \in \Nc_i}n_j\tau$ otherwise. We
define $E_i \in \real^{k_i \times (m + n)\tau}$ to be the matrix
consisting of all ones and zeros such that $E_iw^d = w^d_i$ and $E_iw
= w_i$.

\subsection{Data-Based Representation of Network Trajectories}

Lemma~\ref{lem:fundamental} states conditions under which the behavior
of the system can be described completely by the Hankel matrix of the 
sampled data. Here we extend Lemma
\ref{lem:fundamental} to the setting of a network system to build a 
data-based representation of network trajectories using the Hankel 
matrices of the data available to each agent, $\Hc_\tau(w^d_i)$. We show that
under certain conditions the image of $\Hc_\tau(w^d_i)$ is the set of all 
possible trajectories of node $i$. 

\begin{proposition}\longthmtitle{Sufficiency of Date-Based Image
    Representation}\label{prop:sufficient}
  If for each $i \in \Vc$ there exists $g_i \in \real^{T_d - \tau + 1}$
  with $\Hc_\tau(w^d_i)g_i = w_i$, then $w$ is a trajectory
  of~\eqref{eq:dynamics}.
\end{proposition}
\begin{proof}
  Writing $g_i = (g_i(0), g_i(1), \dots, g_i(T_d - \tau))^\trans$ so for
  all $0 \leq t < \tau - 1$, $w_i(t) = \sum_{k=0}^{T_d - \tau +
    1}g_i(k)w^d_i(t + k)$, it follows that $x_i(t + 1) = \sum_{k=0}^{T
    - \tau} g_i(k)x^d_i(t + k + 1)$ so for~$i \notin S$,
  \begin{align*}
    x_i(t + 1) &=\sum_{k=0}^{T_d - \tau}g_i(k)\Big( A_{ii}x^d_i(t + k) +
    \sum_{j \in \Nc_i}A_{ij}x^d_j(t + k) \Big) 
    \\ 
    &=A_{ii}x_i(t) + \sum_{j \in \Nc_i}A_{ij}x_j(t).
  \end{align*}
  By a similar computation, we can show that for each $i \in S$,
  \begin{align*}
    x_i(t + 1) = A_{ii}&x_i(t) + \sum_{j \in \Nc_i}A_{ij}x_j(t) + B_{i}u_{i}(t),
  \end{align*}
  which is consistent with \eqref{eq:dynamics}.
\end{proof}

Next we identify conditions for the converse of the above result to
hold, i.e., when the Hankel matrices of all the agents characterize
all possible network trajectories.

\begin{proposition}\longthmtitle{Necessity of Data-Based Image
    Representation}\label{prop:network} 
  If $(A, B)$ is controllable, $w^d$ is a trajectory of
  \eqref{eq:dynamics} and either
  \begin{enumerate}
  \item $u^d$ is persistently exciting of order $n +
    \tau$ \label{case:apriori};
  \item $\col(u^d_{i}(t), x_{\Nc_i}^d(t))$ is persistently exciting
    of order $n_i + \tau$ for each $i \in S$, and $x_{\Nc_i}^d$ is
    persistently exciting of order $n_i + \tau$ for each
    $i \in \Vc \setminus S$; \label{case:aposteriori}
  \end{enumerate}
  then for all $i \in \Vc$ there exists $g_i \in \real^{T_d - \tau + 1}$
  such that $\Hc_\tau(w^d_i)g_i = w_i$.
\end{proposition}
\begin{proof}
  In the case of \ref{case:apriori} we simply apply
  Lemma~\ref{lem:fundamental} to obtain $g \in \real^{T_d - \tau + 1}$
  where $\Hc_\tau(w^d)g = w$, and note that for all $i \in \Vc$, $w_i
  = E_iw = E_i\Hc_\tau(w^d)g = \Hc_\tau(w_i^d)g$, so the result
  follows by letting $g_i = g$.  For case \ref{case:aposteriori}, we
  think of $x_j$ for $j \in \Nc_i$ as an input to node~$i$. Letting $k
  = |\Nc_i|$, where $\Nc_i = \{j_1, j_2, \dots, j_{k} \}$, and
  defining
  \begin{align*}
    \small
    \tilde{B}_i = \begin{cases}
      \begin{bmatrix}
        A_{ij_1} & A_{ij_2} & \cdots & A_{ij_k} & B_{i}
      \end{bmatrix}  & i \in S \\[10pt]
      \begin{bmatrix}
        A_{ij_1} & A_{ij_2} & \cdots & A_{ij_k}
      \end{bmatrix} & i \notin S
    \end{cases},
  \end{align*}
  we have
  \begin{align*}
    \small
    x_i(t + 1) = \begin{cases}
      A_{ii}x_i(t) + \tilde{B}_i\col(x_{j_1}, x_{j_2}, \cdots,
      x_{j_k}, u_{i}) &i \in S \\ 
      A_{ii}x_i(t) + \tilde{B}_i\col(x_{j_1}, x_{j_2}, \cdots,
      x_{j_k}) &i \notin S
    \end{cases}
  \end{align*}
  Let $x^0_i \in \real^{n_i}$ be arbitrary, and $x^0 \in \real^n$ such
  that the $i$th block component is $x^0_i$. Since $(A, B)$ is
  controllable there exists an input $\bar{u}:[0, n] \to \real^m$ such
  the corresponding state trajectory $\bar{x}:[0, n] \to \real^m$ with
  $\bar{x}(0) = x^0$ has $\bar{x}(n) = 0$. Note that if $i \in S$,
  then $\bar{x}_i$ is the state trajectory corresponding to the input
  $\col(\bar{u}_{i}, \bar{x}_{\Nc_i})$, and $\bar{x}_i(0) = x_i^0$
  and $\bar{x}_i(n) = 0$. Likewise, if $i \notin S$, $\bar{x}_i$ is
  the state trajectory corresponding to the input $\bar{x}_{\Nc_i}$,
  and $\bar{x}_i(0) = x_i^0$ and $\bar{x}_i(n) = 0$. Hence $(A_{ii},
  \tilde{B}_i)$ is controllable for all $i \in \Vc$ and the result
  follows from Lemma \ref{lem:fundamental}.
\end{proof}

\begin{remark}\longthmtitle{Feasibility of Identifiability Conditions}
  \label{rem:feas}
  {\rm Proposition \ref{prop:network} gives conditions on when the
    data is rich enough to characterize all possible trajectories of
    the system. Condition (i) gives conditions on the input sequence,
    $u^d$, which guarantee \textit{a priori} the identifiability of
    the system from data.  This condition is generically true in the
    sense that the set of sequences $u^d$ which are not persistently
    exciting of order $n + \tau$ (even though for all $i \in S$,
    $u^d_i$ is) have zero Lebesgue measure. In general, it is
    difficult to verify condition (i) in a distributed manner.  On the
    other hand, it is straightforward to verify condition~(ii) using
    only information available to the individual agents. However this
    verification must be done in an ad hoc manner, after the input has
    been applied to the system. While the condition is sufficient for
    identifiability, there are systems where for all inputs $u^d$, the
    resulting trajectory $w^d$ will never satisfy it.  \oprocend}
\end{remark}

\subsection{Equivalent Network Optimization
  Problem}\label{subsec:network_opt}
Here, we build on the data-based image representation of network
trajectories in a distributed fashion to pose a network optimization
problem that can be solved with the data available to each agent,
which is equivalent to an LQR problem with a time horizon of $T$. Each
node can use this representation along with $T_\text{ini} > 0$ past
states and inputs to predict future trajectories assuming that the
hypotheses of Proposition~\ref{prop:network} are satisfied.  Formally,
let $\tau = T_\text{ini} + T + 1$ and let $u^\text{ini}:[0,
T_\text{ini} - 1] \to \real^m$ and $x^\text{ini}:[0, T_\text{ini} - 1]
\to \real^n$ be sequences such that $\col(u^\text{ini}, x^\text{ini})$
is a $T_\text{ini}$ long trajectory of the system. In the network
optimization we introduce below, we optimize over system trajectories
$\col(u, x)$ of length $\tau$, constrained so the first $T_\text{ini}$
samples of $u$ and $x$ are $u^\text{ini}$ and $x^\text{ini}$
resp. This plays a similar role to the initial condition constraint
in~\eqref{eq:lqr}.

For each node $i \in \Vc$, define
\begin{equation}\label{eq:constraints}
  H_i = \begin{bmatrix} \Hc_\tau(u^d_{i}) \\ \Hc_\tau(x^d_{\Nc_i})
    \\ \Hc_\tau(x^d_i)
  \end{bmatrix} \text{ if } i\in S \text{ and } H_i = \begin{bmatrix}
    \Hc_\tau(x^d_{\Nc_i}) \\ \Hc_\tau(x^d_i)
  \end{bmatrix} \text{ if } i \notin S .
\end{equation}
Consider the following problem
\begin{align}\label{eq:dlqr}
  & \underset{g_i, u_{i}, x_i}{\text{minimize}} & & \sum_{i \in S}
  J_i(x_i, u_{i}) + \sum_{i \in \Vc \setminus S} J_i(x_i)
  \\
  \notag & \text{subject to} & & H_ig_i = \col( u^\text{ini}_{i},
  u_{i}, x^\text{ini}_{\Nc_i}, x_{\Nc_i}, x^\text{ini}_i, x_i ), &i
  \in S
  \\
  \notag & & & H_ig_i = \col(x^\text{ini}_{\Nc_i},
  x_{\Nc_i},x^\text{ini}_i, x_i), &i \notin S&
  \\
  \notag & & & x_i(T) = 0 &i \in V.
\end{align}
Although \eqref{eq:dlqr} does not necessarily have a unique optimizer,
any optimizer $(g^*, u^*, x^*)$ of~\eqref{eq:dlqr} is such that $u^*$
and $x^*$ are the optimal input and state sequences of~\eqref{eq:lqr},
as formalized~next.

\begin{proposition}\longthmtitle{Equivalent Network
    Optimization}\label{prop:equivalence}
  Consider the system~\eqref{eq:dynamics} and sample trajectory $w^d$
  satisfying the hypotheses of Proposition \ref{prop:network} and let
  $x^0 = Ax^\text{ini}(T_\text{ini} - 1) + Bu^\text{ini}(T_\text{ini}
  - 1)$.  Then the following hold:
  \begin{enumerate}
  \item If problem~\eqref{eq:lqr} is feasible, then it has a unique
    optimizer;
  \item Problem \eqref{eq:dlqr} is feasible if and only
    if~\eqref{eq:lqr} is feasible;
  \item If \eqref{eq:dlqr} is feasible, $(u^{1 ,*}, x^{1, *})$ is the
    optimizer of \eqref{eq:lqr} and $(g^*, u^{2, *}, x^{2, *})$ is an
    optimizer of \eqref{eq:dlqr}, then $u^{1, *} = u^{2, *}$ and
    $x^{1, *} = x^{2, *}$.
  \end{enumerate}
\end{proposition}

We omit the proof for space reasons, but note that it is the
  analogue of Theorem 5.1 and Corollary 5.2 in~\cite{JC-JL-FD:19} for
  the case of network systems once one invokes
  Propositions~\ref{prop:sufficient} and~\ref{prop:network} above.
Unlike the original network optimization problem~\eqref{eq:lqr}, for
which agents lack knowledge of the system matrices $A$, $B$, the
network optimization problem~\eqref{eq:dlqr} can be solved in a
distributed way with the information available to them.  The structure
of the problem (aggregate objective functions plus locally expressible
constraints) makes it amenable to a variety of distributed
optimization algorithms, see
e.g.,~\cite{SB-NP-EC-BP-JE:11,AC-BG-JC:17}.  In Section~\ref{sec:pd}
below, we employ a primal-dual dynamic to find asymptotically a
solution of~\eqref{eq:dlqr} in a distributed way.

\begin{remark}\longthmtitle{Scalability of Network
    Optimization}\label{rem:scalability}
  {\rm As the number of nodes in the network increases so does the
    state dimension, hence more data is required in order to maintain
    persistency of excitation. A necessary condition is $T_d \geq (n +
    m + 1)(T_\text{ini} + T) - 1$.  Assuming that $T \sim O(n)$,
    $T_\text{ini} \sim O(1)$, we have $T_d \sim O(n^2 +
      mn)$. The decision variable for each node is $z_i = \col(g_i,
    u_i, x_i)$ when $i \in S$ and $z_i = \col(g_i, x_i)$ otherwise.
    The size of $z_i$ is $O(n^2 + mn)$.  However, using the
    distributed optimization algorithm of Section~\ref{sec:pd}, agent
    $i$ only needs to send messages of size $O(k_i)$ to agent
      $j$.}  \oprocend 
\end{remark}

\section{Distributed Data-Based Predictive Control}
Here we introduce a distributed data-based predictive control scheme
to stabilize the system \eqref{eq:dynamics} to the origin, as
described in Section~\ref{sec:problem}. To do this, we solve
the network optimization problem~\eqref{eq:dlqr} in a receding horizon
manner with $u^\text{ini}$ and $x^\text{ini}$ updated every time step
based on the systems current state. The control scheme is summarized
in Algorithm~\ref{alg:dmpc}.

\begin{algorithm}[h]
  \caption{Distributed Data-Based Predictive Control}\label{alg:dmpc}
  \begin{algorithmic}[1]
  \State \algorithmicrequire ~Sample trajectory $w^d$, performance
  indices $(Q_i)_{i=1}^N$, $(R_i)_{i=1}^N$
  \State Initialize $H_i$ as in equation \eqref{eq:constraints}, let 
  $u^\text{ini}$ and $x^\text{ini}$ be the~$T_\text{ini}$ most recent 
  states and inputs respectively, and set $t = 0$.

  \While {$\norm{x} > 0$} \State Use a distributed optimization
  algorithm to obtain an approximate solution to \eqref{eq:dlqr},
  $\hat{z} = \col(\hat{g}, \hat{u}, \hat{x})$, such that
  $\lVert\hat{u} - u^*\rVert \leq \delta \min \{1, \lVert
  x^\text{ini}(T_\text{ini} - 1)\rVert \}$.
  \label{step:opt}
  \State Apply the input $\hat{u}(0)$. 
  \State Set $t$ to $t +1$ and $u^\text{ini}$ and $x^\text{ini}$ to 
  the $T_\text{ini}$
  most recent inputs and states respectively. 
  \EndWhile
  \end{algorithmic}
\end{algorithm}

The rest of the section proceeds by first showing that the controller
resulting from Algorithm~\ref{alg:dmpc} is stabilizing even when the
network optimization~\eqref{eq:dlqr} is solved only approximately; and
then introducing a particular distributed solver for~\eqref{eq:dlqr}
along with a suboptimality certificate to check, in a distributed
manner, the stopping condition in Step~\ref{step:opt} of
Algorithm~\ref{alg:dmpc}. 

\subsection{Stability Analysis of Closed-Loop System}
In the rest of the paper, we rely on the following assumption.

\begin{assumption}\label{a:stab_assumptions}
    Consider the system~\eqref{eq:dynamics},
  \begin{enumerate}
  \item the collected data satisfies the hypotheses of
    Proposition~\ref{prop:network}; \label{assumption}
  \item the optimization problem~\eqref{eq:dlqr} is feasible for all
    $(u^\text{ini}, x^\text{ini})$ which is a valid
    $T_\text{ini}$-sample long system trajectory.
  \end{enumerate}
\end{assumption}

Since the system is controllable (cf. \emph{(i)}), a sufficient
condition for guaranteeing the feasibility in \emph{(ii) is $T \geq
  n$}. In fact, in such case, there exists a trajectory $(u, x)$ such
that $x(T) = 0$. It follows that $(u, x)$ is feasible for
\eqref{eq:lqr}, so by Proposition \ref{prop:equivalence}, there exists
some $g$ such that $(g, u, x)$ is feasible for \eqref{eq:dlqr}.

Under Assumption~\ref{a:stab_assumptions}, the closed-loop
system with the controller corresponding to a receding horizon
implementation of~\eqref{eq:lqr} is globally exponentially stable,
cf.~\cite[Theorem 12.2]{FB-AB-MM:17}.  Unlike \cite{JC-JL-FD:19}, we
do not assume we have access to the exact solution of \eqref{eq:dlqr}
since distributed optimization algorithms typically only converge
asymptotically to the true optimizer and must be terminated in finite
time. Here we show that Algorithm~\ref{alg:dmpc} still stabilizes the
system when the tolerance~$\delta$ is sufficiently small.

\begin{theorem}\longthmtitle{Distributed, Data-Based Predictive
    Control is Stabilizing}\label{thm:stability}
  For $\delta > 0$, let $\phi_\delta:\real^n \to \real^m$ be the
  feedback control corresponding to Algorithm~\ref{alg:dmpc}.  Under
  Assumption~\ref{a:stab_assumptions}, there exists $\delta^* > 0$
  such that for all $\delta < \delta^*$, the origin is globally
  asymptotically stable with respect to the closed-loop dynamics $x(t
  + 1) = Ax(t) + B\phi_\delta(x(t))$.
\end{theorem}
\begin{proof}
  Let $\phi_\text{mpc}:\real^n \to \real^m$ be the feedback
  corresponding to a receding horizon implementation of
  \eqref{eq:lqr}.  Consider the system $ x(t + 1) = f(x(t), v(t)), $
  where $ f(x, v) = A + B\phi_\text{mpc}(x) + Bv$.  Let $J^*(x^0) =
  J(x^*, u^*)$, where $( u^*, x^*)$ is an optimizer of
  \eqref{eq:lqr}. Because \eqref{eq:dlqr} is nondegenerate, $J^*$ is
  continuously differentiable, cf.~\cite[Theorem~6.9]{FB-AB-MM:17},
  and the system is input-to-state stable (ISS) with $J^*$ being a
  ISS-Lyapunov function satisfying $ J^*(f(x, v)) - J^*(x) \leq
  -\alpha \norm{x}^2 + \sigma \norm{v}^2 $ for constants $\alpha,
  \sigma > 0$, cf. \cite[Theorem 1]{AJ-ASM:02} (albeit the result is
  stated there for systems where $A$ is Schur stable, the same
  reasoning is valid when there are no state or input constraints and
  $A$ is unstable).  Because the system $x(t + 1) = f(x(t), 0)$
  without disturbances is exponentially stable
  \cite{DQM-JBR-CVR-POMS:00}, it follows by \cite[proof of Lemma
  3.5]{ZPJ-YW:01} that the gain function is linear, so there exist
  $\gamma > 0$ and a class $\Kc\Lc$ function $\beta$ such~that
  \[
  \norm{x(t)} \leq \beta(\norm{x(0)}, t) + \gamma \sup_{0 \leq \tau
    \leq t} \norm{v(\tau) },
  \]
  for all $t \in \intnonneg$.  Let $x:\intnonneg \to \real^n$ be a
  trajectory of the closed-loop dynamics of \eqref{eq:dynamics} with
  the controller described by Algorithm~\ref{alg:dmpc} where $\delta <
  \delta^*= \gamma^{-1}$ and define $v(t) = \phi_\delta(x(t)) -
  \phi_\text{mpc}(x(t))$. It follows that $x(t + 1) = f(x(t), v(t))$.
  Note that $\phi_\text{mpc}(x(t)) = u^*(0; u^\text{ini},
  x^\text{ini})$, where $u^\text{ini} = \phi_\delta(x(t - 1))$ and
  $x^\text{ini} = x(t - 1)$ so it follows~that $ \norm{v(t)} =
  \norm{\hat{u}(0; u^\text{ini}, x^\text{ini}) - u^*(0; u^\text{ini},
    x^\text{ini})} \leq \delta \norm{x(t - 1)}.  $ We claim that for
  all $k \in \intpos$, there exists $T_k \in \intpos$ such that
  $\norm{x(t)} \leq (k + 1)(\gamma\delta)^k$ whenever $t \geq
  T_k$. The case when $k=1$ follows by observing that $\norm{x(t)}
  \leq \beta(\norm{x(0)}, t) + \gamma\delta$ and there exists $T_1$
  such that $\beta(\norm{x(0)}, t) < \gamma\delta$ for all $t \geq
  T_1$. If the claim holds for some $k$, then for all $t \geq T_k +
  1$,
  \begin{align*}
    \norm{x(t)}
    &\leq \beta((k + 1)(\gamma\delta)^k, t) + \gamma \sup_{T_k + 1
      < \tau \leq t} \norm{v(\tau) }\\
    & \leq \beta((k + 1)(\gamma\delta)^k, t) +
      (k + 1)(\gamma\delta)^{k + 1},
  \end{align*}
  so by choosing $T_{k + 1}$ such that
  $\beta((k + 1)(\gamma\delta)^k, t) < (\gamma\delta)^{k + 1}$ for all
  $t \geq T_{k + 1}$, then
  $\norm{x(t)} < (k + 2)(\gamma\delta)^{k + 1}$ for all
  $t > T_{k + 1}$ and the claim follows by induction. Hence
  $ \limsup_{t \to \infty} \norm{x(t)} \leq \limsup_{k \to \infty} (k
    + 1)(\gamma\delta)^{k} = 0. $ To show global Lyapunov stability,
  let $\eta > 0$ be arbitrary, and suppose that $\norm{x(0)}$ is
  chosen so that $\beta(\norm{x(0)}, 0) < (1 - \gamma\delta)\eta$ and
  $\norm{x(0)} < \eta$. Then for all~$t > 0$,
  \begin{align*}
    \norm{x(t)} &\leq (1 - \gamma\delta)\eta + \gamma \sup_{0 \leq \tau
    \leq t} \norm{v(\tau) } \\
    &\leq (1 - \gamma\delta)\eta + \gamma\delta
  \norm{x(t - 1)}.
  \end{align*}
   If $\norm{x(t - 1)} < \eta$, then
  $\norm{x(t)} < \eta$. It follows by induction on $t$ that
  $\norm{x(t)} < \eta$ for all $\eta$.
\end{proof}

\subsection{Primal-Dual Solver for Network Optimization}\label{sec:pd} 

In this section we introduce a method for solving the optimization
problem \eqref{eq:dlqr} in a distributed way. 
We let
  \[ z_i = \begin{cases}
    \col(g_i, u_i, x_i) & i \in S, 
    \\
    \col(g_i, x_i) & i \notin S,
   \end{cases}
\]
and $z = \col(z_1, \dots, z_N)$. Note $z_i \in \real^{d_i}$, where
$d_i = T - (T_\text{ini} + T) + 1 + n_i + m_i$ for
$i \in S$ and $d_i = T - (T_\text{ini} + T) + 1 + n_i$
otherwise. Problem~\eqref{eq:dlqr} can be written~as
\begin{align}\label{eq:dqp}
  \underset{z_i \in \real^{d_i}}{\text{minimize}}\qquad & \sum_{i \in
    \Vc} \norm{z_i}_{\Qc_i}^2
  \\
  \notag \text{subject to}\qquad
  & \Ac_i z_{\Nc_i} = b_i.
\end{align}
for suitable $\Qc_i \in \real^{d_i \times d_i}$, $\Qc_i \succeq 0$,
$\Ac_i \in \real^{c_i \times d_i}$,, and $b_i \in \real^{c_i}$, with
$i, c_i \in \integers$.  The Lagrangian of~\eqref{eq:dqp}~is
\begin{align*}
  \Lc(z, \lambda) = \sum_{i \in \Vc} \norm{z_i}_{\Qc_i}^2 +
  \lambda_i^\trans(\Ac_i z_{\Nc_i} - b_i).
\end{align*}
If $\lambda^*$ is an optimizer of the dual problem, then the pair
$(z^*, \lambda^*)$ is a (min-max) saddle point of $\Lc$, meaning that
$ \Lc(z^*, \lambda) \leq \Lc(z^*, \lambda^*) \leq \Lc(z, \lambda^*)$
for all $z \in \real^d$ and $\lambda \in \real^c$.  
The saddle-point
property of the Lagrangian suggests that the primal-dual flow, which
descends along the gradient of the primal variable and ascends along
the gradient of the dual variable,
\begin{equation}\label{eq:pdflow}
  \begin{aligned}
    \begin{bmatrix}
      \dot{z}_i \\ \dot{\lambda_i} 
    \end{bmatrix}  
    &= \begin{bmatrix}
      -\nabla_{z_i}\Lc(z, \lambda, \mu) \\ 
      \nabla_{\lambda_i}\Lc(z, \lambda, \mu) 
    \end{bmatrix}  \\
    &= \begin{bmatrix} -2\Qc_i z_i - F_{ii}\Ac_i^\trans
      \lambda_i - \sum_{j \in \Nc_i}
      F_{ij}\Ac_j^\trans \lambda_j \\
      \Ac_i z_{\Nc_i} + b_i 
    \end{bmatrix},
  \end{aligned}
\end{equation}
can be used to compute the optimizer. Here, $F_{ij} \in \real^{(d_i +
  \sum_{j\in\Nc_i}d_j) \times d_i}$ is the matrix such
that~$F_{ij}z_{\Nc_j} = z_i$.  By \cite[Corollary 4.5]{AC-BG-JC:17},
the flow converges asymptotically to a saddle point of~$\Lc$. This
procedure is fully distributed, since the flow equations
in~\eqref{eq:pdflow} can be computed with the information available to
each agent or its direct neighbors. In particular, if $j \in \Nc_i$,
then the message agent $j$ shares with agent $i$ consists of
$\col(x_j, \lambda_j) \in \real^{n_jT + k_j}$, which is $O(k_i)$
(cf. Remark~\ref{rem:scalability}).

We conclude by providing a certificate that can be used to verify the
stopping condition of Step~\ref{step:opt} in Algorithm~\ref{alg:dmpc}.

\begin{proposition}\longthmtitle{Suboptimality Certificate}
  \label{prop:cert}
  Let $u^*$ and $x^*$ denote the optimal input and state
    trajectories of \eqref{eq:lqr}, $y = \col(z, \lambda)$, $\Qc =
  \diag(\Qc_1, \Qc_2, \dots, \Qc_N)$, $\Ac = [F_1^\trans \Ac_1^\trans,
  F_2^\trans \Ac_2^\trans, \dots, F_N^\trans \Ac_N^\trans]^\trans$, $b
  = \col(b_1, b_2, \dots, b_N)$, and
  \begin{align*}
    M = \begin{bmatrix}
      -2\Qc^\trans & -\Ac^\trans  \\ \Ac & 0  
    \end{bmatrix} \qquad q = \begin{bmatrix}
      0 \\ b 
    \end{bmatrix}.
  \end{align*}
   Under Assumption \ref{a:stab_assumptions}, and with the flow
  given by~\eqref{eq:pdflow}, if $\lVert{\col(\dot{z}_i,
    \dot{\lambda}_i)}\rVert < \rho$ for all $i \in \Vc$, where $\rho =
  \frac{\delta^2}{N\norm{M^\dagger}^{2}}$, then $\norm{u - u^*} <
  \delta$.
\end{proposition}
\begin{proof}
  The set of saddle points of $\Lc$ is~$\Sc = \setdef{y}{My + q = 0
  }$.  Since the optimal input and state trajectories are unique, all
  saddle points share the property that, for each $i \in S$, the
  $(u_i,x_i)$ components of their $z_i$ equal $(u^*_i, x^*_i)$.  Given
  an arbitrary $y$, we have $\norm{u - u^*} \leq \norm{y - w}$ for all
  $w \in \Sc$, and hence $ \norm{u - u^*} \leq \inf_{w \in \Sc}\norm{y
    - w}$.  The set of saddle points can also be described as $\Sc =
  \{\hat{y}\} + \ker M$, for~any~$\hat{y} \in \Sc$.  Therefore,
  $\inf_{w \in \Sc}\norm{y - w} = \inf_{v \in \ker M}\norm{y - \hat{y}
    - v}$.  Since~$I - M^\dagger M$ is the orthogonal projection onto
  $\ker M$,
  \begin{align*}
    \inf_{v \in \ker M}&\norm{y - \hat{y} - v}=\norm{(y - \hat{y}) -
      (I - M^\dagger M)(y - \hat{y})}
    \\
    &\hspace{-0.5cm}=\norm{M^\dagger(My - M\hat{y})} = \norm{M^\dagger(My + q)}
    \leq \norm{M^\dagger}\norm{\dot{y}},
  \end{align*}
  and therefore,
  \[ 
  \norm{u - u^*}^2 \leq \norm{M^\dagger}\sum_{i \in
    \Vc}\norm{\dot{y}_i}^2 < \norm{M^\dagger}\sum_{i \in
    \Vc}\frac{\delta^2}{N\norm{M^\dagger}} = \delta^2.
  \]
\end{proof}

The suboptimality certificate can be checked in a fully distributed
manner using information locally available to each agent provided that
$\norm{M^\dagger}$ is known. Because $M$ depends only on the objective
$\Qc$ and constraints~$\Ac$, which in turn comes from the sample
trajectory $w^d$, it can be computed offline. Finally, it is possible
for each agent to compute a bound on $\norm{M^\dagger}$ using the fact
that, for $M = [ M_1^\trans, M_2^\trans, \dots M_N^\trans ]^\trans$,
one has $\norm{M^\dagger } \leq \lVert M_i^\dagger \rVert$ for all $i
\in \Vc$.  It follows that for all~$i \in \Vc$,
\[
\norm{M^\dagger} \leq \norm{ \begin{bmatrix} 2\Qc_i &
    F_{ii}\Ac_i^\trans & F_{ij_1}\Ac_{j_1}^\trans & \cdots &
    F_{ij_{|\Nc_i|}}\Ac_{j_{|\Nc_i|}}^\trans
    \end{bmatrix}^\dagger },
\]
for $ \{j_1, j_2, \dots, j_{|\Nc_i|} \} = \Nc_i$, so each agent can
compute a bound on $\norm{M^\dagger}$ using data available to itself
and its neighbors.

\begin{figure}[t!]
  \centering%
  \subfigure[Simulation of  Newman-Watts-Strogatz network with~$T = n$]{
    \hspace{-1.5ex}\includegraphics[width=.49\linewidth]{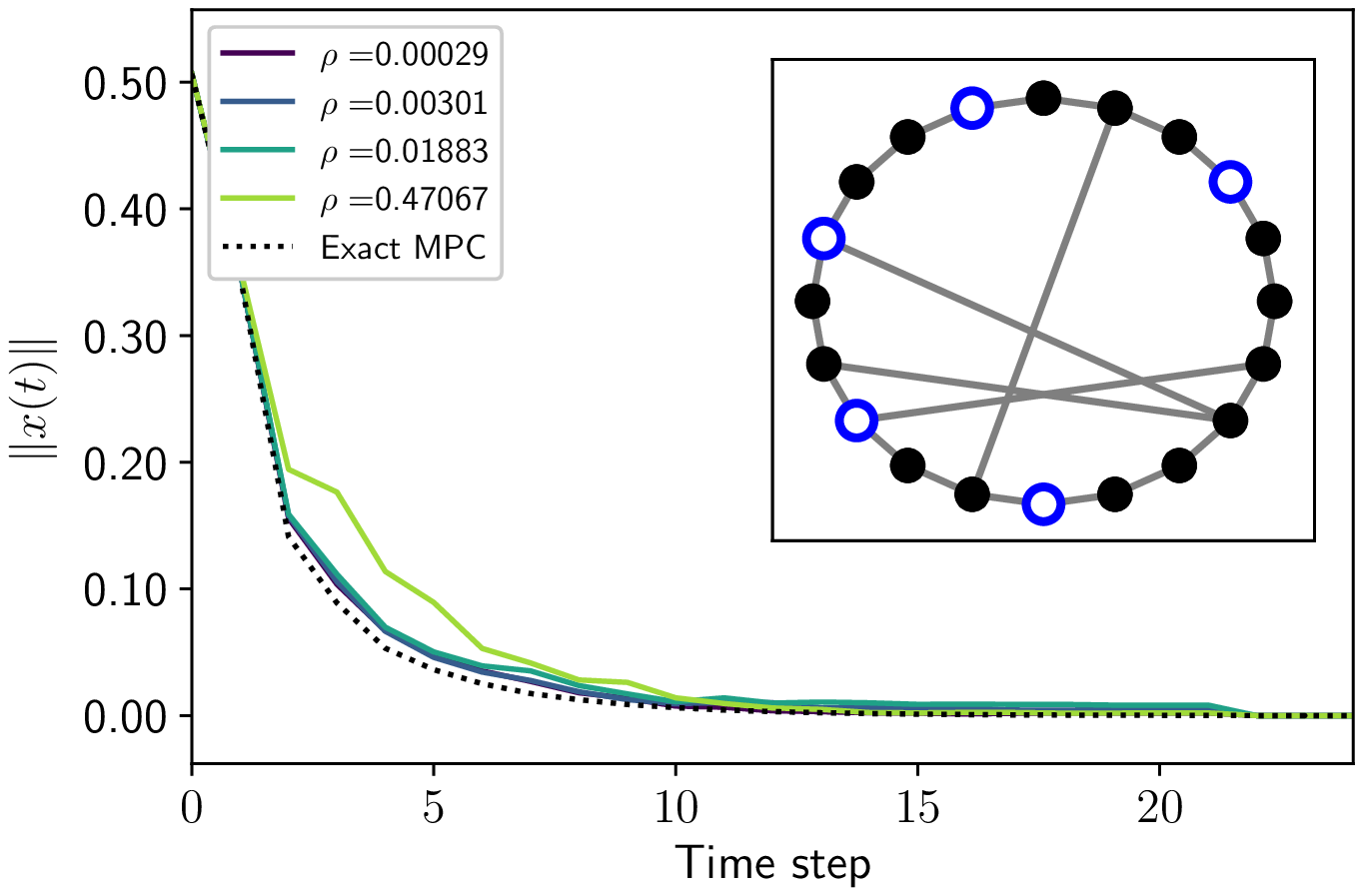}\label{fig:largeT} 
    \includegraphics[width=.49\linewidth]{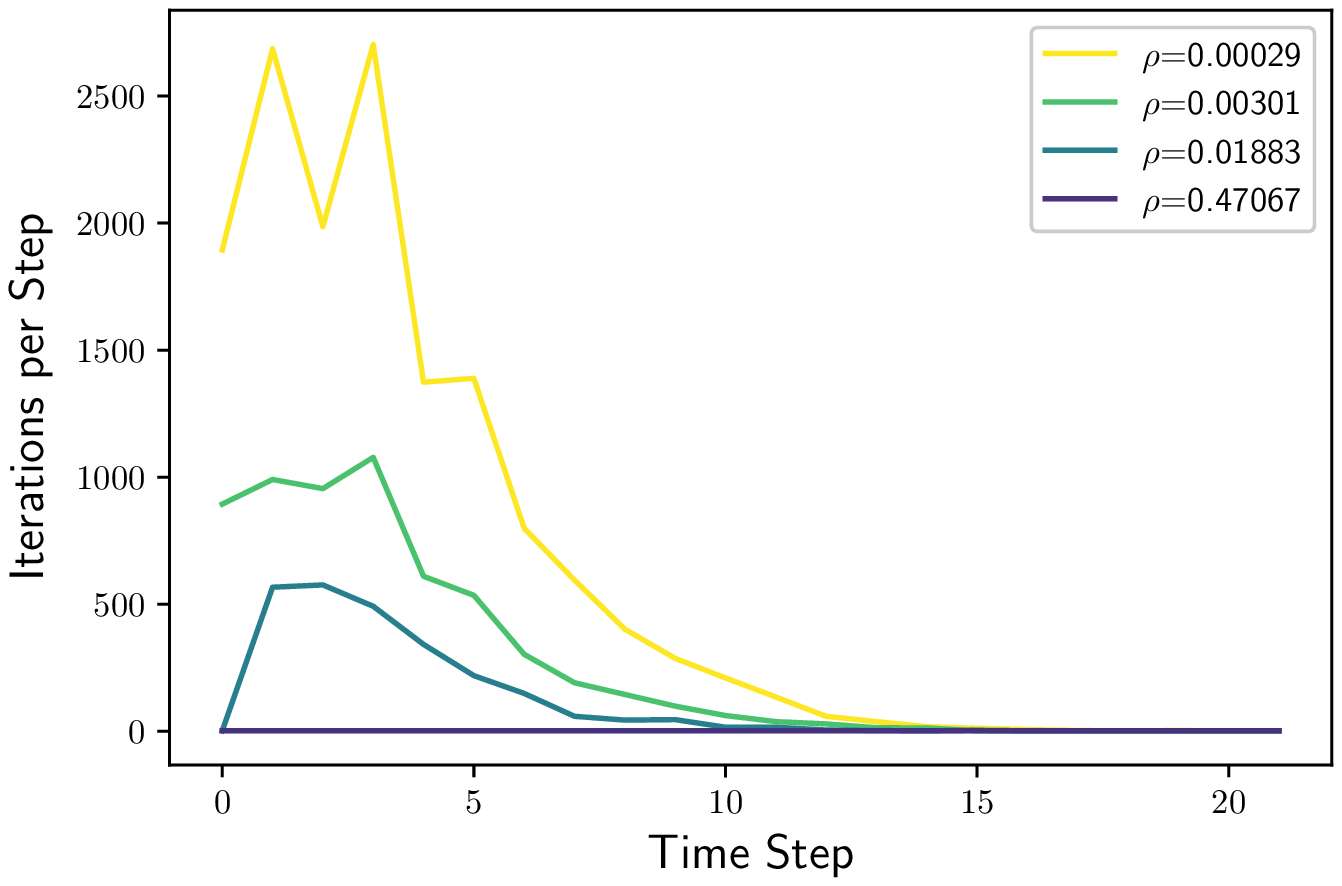}
    \label{fig:complex_iter}
    }
  \\
  \subfigure[Simulation of star network with $T = 5 < n$]{
    \hspace{-1.5ex}\includegraphics[width=.49\linewidth]{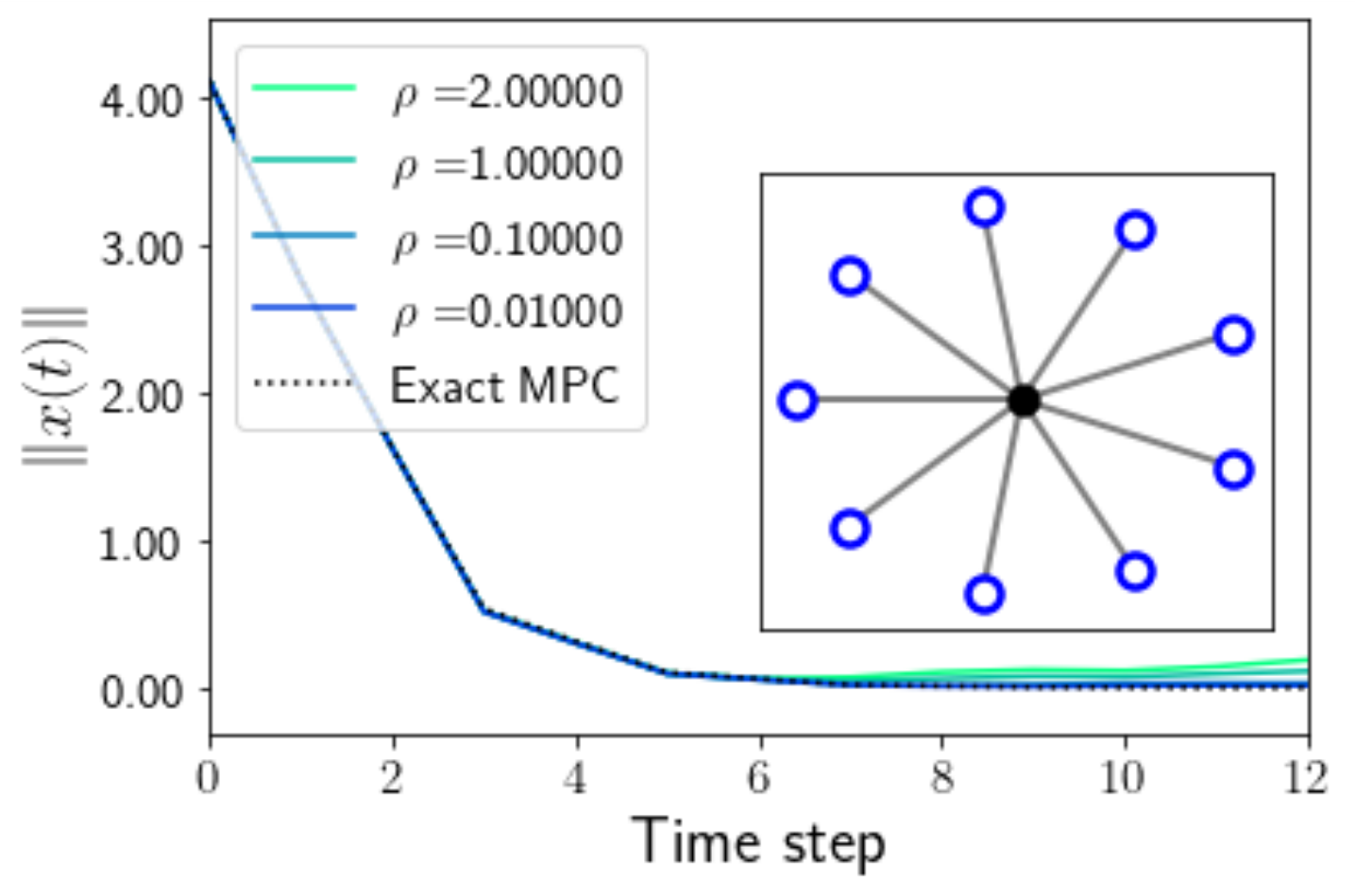}\label{fig:smallT}
    \includegraphics[width=.49\linewidth]{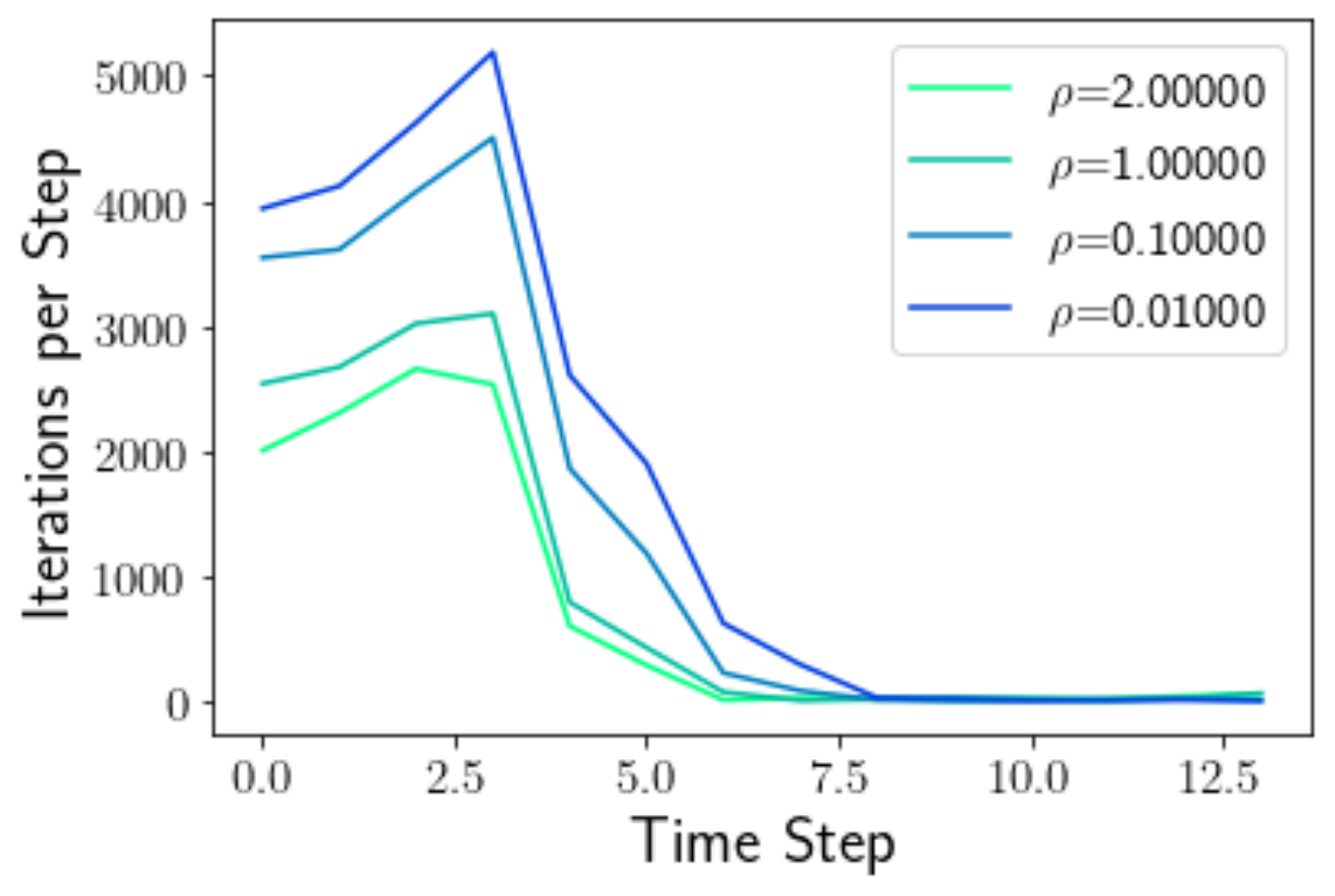}\label{fig:star_iter}
  }
  \caption{State trajectories of Data-based Receding Horizon
    Controller on various network topologies.}\label{fig:sim}
\vspace*{-2.5ex}
\end{figure}

\subsection{Numerical Simulations}
We simulate the proposed distributed data-based predictive controller
on a Newman-Watts-Strogatz network \cite{MEN-DJW:99} and a star
network. In each case, $A$ and $B$ are chosen at random so that $(A,
B)$ is controllable.  The input sequence is chosen as $u^d(t) =
Kx^d(t) + w(t)$, where $K$ is a matrix so that $A + BK$ is marginally
stable (the data does not need to be generated from a stable system,
but this is done to avoid numerical issues), and $w(t)$ is a Gaussian
white noise process. We use Proposition~\ref{prop:network} to ensure
that the data is informative enough for data-driven control. In both
cases, condition~\ref{case:apriori} is satisfied.
Condition~\ref{case:aposteriori} fails for the Newman-Watts-Strogatz
network, but is satisfied by the star network
(cf. Remark~\ref{rem:feas}).  We integrate the primal-dual flow using
the stopping condition in Proposition~\ref{prop:cert} is used to
terminate the flow.  Fig.~\ref{fig:sim} shows the closed-loop state
trajectories and the number of iterations on each time step with
different values of $\rho$.  For the Newman-Watts-Strogatz network,
cf.~Fig.~\ref{fig:largeT}, the time horizon is $T = n$.  For the star
network, cf.~Fig.~\ref{fig:smallT}, the time horizon is $T = 5 < n$,
but the optimization at each time step is still feasible.  In both
cases, the distributed data-based predictive controller better
approximates the exact MPC for smaller values of~$\rho$ at the cost of
more iterations per time~step.

\section{Conclusions and Future Work}\label{sec:conclusion}
We have introduced a distributed data-based predictive controller for
stabilizing network linear dynamics described by unknown system
matrices. Instead of building a dynamic model, agents learn a
non-parametric representation based on a single trajectory and use it
to implement a controller as the solution of a network optimization
problem solved in a receding horizon manner and in a distributed way.
Future work will explicitly quantify the tolerance $\delta^*$ in terms
of the available data and study ways to construct a terminal cost
without knowledge of the underlying model to guarantee stability when
the stabilizing terminal constraint is omitted.  We plan to extend the
results to cases where there are constraints on the state and input,
characterize the robustness properties of the introduced control
scheme, investigate ways of improving its scalability, and consider
more general scenarios, including the presence of noise in the data,
inputs not persistently exciting of sufficiently high order, and
partial observations of the network state. We also plan to explore
improvements to the primal-dual flow to solve the optimization problem
with fewer iterations and less communication between the agents.

\end{document}